\documentclass[10pt,draft,reqno]{amsart}
     \makeatletter
     \def\section{\@startsection{section}{1}%
     \z@{.7\linespacing\@plus\linespacing}{.5\linespacing}%
     {\bfseries%\normalfont\scshape
     \centering
     }}
     \def\@secnumfont{\bfseries}
     \makeatother
\newtheorem{theorem}{Theorem}[section]
\newtheorem{lemma}[theorem]{Lemma}

\theoremstyle{definition}

\theoremstyle{remark}
\newtheorem{remark}[theorem]{Remark}
\numberwithin{equation}{section}
\begin{document}

\title[Optimal Control For Absorbing Regimes]{Optimal Control Problems for Stochastic processes with absorbing regime}

\author{Yaacov Kopeliovich}
\address{Yaacov Kopeliovich: Finance Department, University of Connecticut, Storrs, CT 06269, USA}
\email{yaacov.kopeliovich@uconn.edu}
%\urladdr{http://www.math.univ.edu/$\sim$johndoe \bf(optional)}

%\thanks{* This research is supported by the Doe Foundation}


\begin{abstract}
In this paper we formulate and solve an optimal problem for Stochastic process with a regime absorbing state. The solution for this problem is obtained through a system of partial differential equations. The method is applied to obtain an explicit solution for the Merton portfolio problem when an asset has a default probability in case of a $\log$ utility.
\end{abstract}

\maketitle

%\noindent $\clubsuit$ Note to author: Use 2020 Mathematics Subject Classification.

\section{Introduction}

The Bellman  principle has been established for almost 60 years. For continuous control problems this leads to the HJB equation for the value function which is a candidate for a solution of optimal control problems. The stochastic analogue for Brownian motions is well known  and at least heuristically a simple application of It\^o's lemma. In this paper we consider a problem in which the stochastic process is a regime switching process between $2$ regimes and the second regime is an absorbing one. Such regimes are ubiquitous in finance and economy, examples include:
\begin{itemize}
\item Bankruptcy - the state of bankruptcy is final and absorbing 
\item Bond convertibility - When bonds are converted into stocks. 
\item The state of the economy - For portfolio optimization one can approximate recessions as an absorbing state for short enough time horizons 
\end{itemize}
As an industrial application we can consider linear quadratic control problem under the scenario that there is a finite probability that the factory will be upgraded and thus a new regime of cost structure and production will be initiated. The new state will be an absorbing state for a certain period of time.\par
Optimal control problems with regime switching were considered before. For example in \cite{ZY1} the authors consider Merton's portfolio problem under regime switching models that are characterized by volatility regimes for multiple stocks. They develop a theory of nearly optimal solutions to Merton problem. However the setup in their paper is less constrained than we consider in this paper and hence the theory that is developed is more involved from the set-up we obtain in this paper. One of the authors \cite{ZY2} has also considered a multi-state problem analogous to Markowitz setup in Finance and for this case has obtained an analytical solution for the allocation problem. ( The setup in the latter paper was somewhat simpler as the authors were looking to minimize risk in a multi-stage Markowitz problem.) \par A problem that included the Merton problem for corporate bonds with a bankruptcy option stock and cash was considered by \cite{BeJa}, however the authors hasn't written the HJB equation and considered the final solution only. \par
The regime switching model we consider in this paper is less general than \cite{ZY1}. It's simplified nature enables a straightforward treatment, which is analogous to the usual HJB for a diffusion process ( no regime switching.). We develop the method in the first section of our paper.\par Using Bellman's  principal we break the problem into sub-problems: 
\begin{enumerate}
\item Optimal problem after in the absorbing state of the world 
\item Optimal problem prior to the jump to the absorbing state. 
\end{enumerate}
First we solve the optimal control  problem for the value function in the absorbing state. Then we apply Ito's lemma to write the HJB for the $J^{pre}$ value function prior to the absorbing state. The resulting system of PDE's is similar to  the usual Bellman equation for a regular diffusion process without jumps. 
We comment that while this method is straightforward we haven't presented a formal argument showing that our approach solves the optimal problem with regime changes. 
In the second section we apply the framework developed in the first section to obtain an analytical solution for the Merton's portfolio problem for a $log$ utility function in a presence of bankruptcy default rate of $h.$ While the bankruptcy problem was considered before, the solution we propose haven't appeared in the literature.

The last section summarizes our results and indicates possible directions for future generalizations. 
\section{HJB equation - recap } 
We consider the following process: 
\begin{equation}
\displaystyle \min _{u}\mathbb {E} \left\{\int _{0}^{T}C(t,X_{t},u_{t})\,dt+D(X_{T})\right\}"
\end{equation}
With stochastic process to $(X_{t})_{t\in [0,T]}$ to optimize and $(u_{t})_{t\in [0,T]}$ is the optimal vector we need to find. Assuming that we have the usual Stochastic process ( No jumps) the HJB equation looks like: 
\begin{equation}
\min _{u}\left\{{\mathcal {A}}V(x,t)+C(t,x,u)\right\}=0
\end{equation}
and $\mathcal {A}$ is the Stochastic differentiation operator for the process $X_t$. Consider now the problem above for these kind of processes: 
\begin{equation}
X_t=1_{t\leq \tau}(X_1)_{t}+1_{t\geq \tau}(X_2)_{t}+hdt1_{d\tau}\left[(X_1)_{\tau}+(X_2)_{\tau}-(X_1)_{\tau}\right]
\end{equation}
In words : $X_(t)$ satisfies the following: 
\begin{itemize}
\item The initial state is given by the stochastic process $X_1$
\item The terminal state if it occurs given by the stochastic process $X_2$ 
\item At time $\tau$ the probability of moving from $X_1$ to $X_2$ is given by $hdt$ conditional on the fact that the transition hadn't occurred earlier. 
\end{itemize}
\begin{remark}
In equivalent terms we can consider a regime switching absorbing process among two states given by stochastic processes $X_i$ and with the Markov transition matrix of the form\footnote{I thank George Yin for pointing this fact out} : 
$$
	\begin{bmatrix} 
	1-hdt & hdt \\
	0 & 0 \\
	\end{bmatrix}
	\quad
	$$
\end{remark}

We consider an optimal control problem under this process. We will see that this has a clear economic motivation in the next section. 
Introduce the value function $V(x,t)$ in two pieces: 
\begin{itemize} 
\item $V^{pre}(x,t)$ - the function before the jump 
\item $V^{after}(x,t)$ - the function after the jump 
\end{itemize} 
To solve the problem we need to find $V^{pre}(x,t)$ and $V^{after}(x,t).$ Consider first $V^{after}(x,t)$ in this case the jump already occurred so we are in a regular optimization problem for stochastic process $X_2$ but without jumps. In this case we can apply the usual Ito's lemma to obtain the following differential equation we need to solve: 
\begin{equation}
\min _{u}\left\{{\mathcal {A}_2}V^{after}(x,t)+C(t,x,u)\right\}=0
\end{equation}
and $\mathcal{A}_2$ is the diffusion operator for the stochastic process $X_2.$
Now we consider the corresponding equation for $dV^{pre}.$ We need to incorporate the regime change event into It\^o's lemma. Conditional on the fact that regime change has not happened until time $t$ we consider $2$ outcomes:
\begin{itemize} 
\item Regime change \textbf{has not happened} in the interval $\left[t,t+dt\right]$ with probability $1-hdt$
\item Regime change \textbf{happened} in the interval  $\left[t,t+dt\right]$ with probability $hdt$
\end{itemize}
Taking the expectation of both outcomes we arrive to the following modification of It\o's lemma: 
\begin{equation}
E_t(dV^{pre})=(1-hdt)\mathcal {A}_1 V^{pre}(x,t)+hdt\left[V^{after}(x_t^{after},t)-V^{pre}(x_t,t\right]
\end{equation}
$x_t$ is the value of the process before the jump while $x_t^{after}$ is the value after the jump. Using the fact that, $hdt\mathcal {A}_1=O(dt^2)$  we obtain the following stochastic differential equation for $V^{pre}$
\begin{equation}
\min _{u}\left\{{\mathcal {A}_1}V(x,t)+h\left[V^{after}(x_t^{after},t)-V^{pre}(x_t,t)\right]+C(t,x,u)\right\}=0
\end{equation}
We summarize the previous discussion in the following theorem: 
\begin{theorem}
Let \begin{equation}
	X_t=1_{t\leq \tau}(X_1)_{t}+1_{t\geq \tau}(X_2)_{t}+hdt1_{d\tau}\left[(X_1)_{\tau}+(X_2)_{\tau}-(X_1)_{\tau}\right]
\end{equation}
be a diffusion process with an absorbing regime $X_2$ with probability of $h.$ Let $V(x,t)$ be the value function optimizing: 
\begin{equation}
\displaystyle \min _{u}\mathbb {E} \left\{\int _{0}^{T}C(t,X_{t},u_{t})\,dt+D(X_{T})\right\}
\end{equation}
Then the solution $V(x,t)$ can be broken into two seperate functions $V^{after}(x,t)$ and $V^{pre}(x,t)$ such that each satisfies the following system of partial differential equations: 
\begin{equation}
\min _{u}\left\{{\mathcal {A}_2}V^{after}(x,t)+C(t,x,u)\right\}=0
\end{equation}
and 
\begin{equation}
\min _{u}\left\{{\mathcal {A}_1}V(x,t)+h\left[V^{after}(x_t^{after},t)-V^{pre}(x_t,t)\right]+C(t,x,u)\right\}=0
\end{equation}
\end{theorem}
Where ${A}_2,{A}_1$ are the diffusion operators for the processes $X_1,X_2$
In the next section we apply these formulas to solve a portfolio Merton problem with bankruptcy in a new way. 
\section{Example - stock with a bankruptcy} 
As an example of the outline in the previous section we consider an optimal stock allocation with a probability of bankruptcy $h$ for the time horizon $[0,T].$ For other works in this direction see \cite{KLS} and \cite{BeJa} who appears to be  closest to our approach.

If $r$ is the risk free rate the dynamics of stock is given by:
\begin{equation}
S_t=\tilde{S_t}1_{t<\tau}
\end{equation}
and 
$$\frac{d\tilde{S}}{\tilde{S}}=\left(\mu dt+ \sigma dB\right)$$
 The probability of the stock being bankrupt  in the interval $\left[\tau,\tau+dt\right]$ \textsc{conditional} on the fact that no bankruptcy event occurred until time $\tau$ is $hdt$
Now we like to use the method outlined to find an optimal allocation to the stock with $S_t.$ We show the following lemma: 
\begin{lemma}
$\pi_1$ the optimal stock weight allocation is given by: 
\begin{equation}
\pi_1=\frac{\mu-r-h}{\sigma^2}
\end{equation}
\end{lemma}
\begin{proof}
According to the method outlined, define the value function in $2$ parts. $J^{after}$ and $J^{pre}$
After bankruptcy only cash is present with risk free rate $r.$  Assume the utility of wealth is: $$U(W)=\log 
W.$$ 
If bankruptcy occurred prior to the maturity $T$ we can invest only in cash and hence if bankruptcy event occurs in time $T$ the value of our investment is: $$W\exp(r(t-T)$$ and therefore: 
\begin{equation}
J^{after}(W,t)=r(t-T)+\log W
\end{equation} 
At time $t$ we have that the probability to go bankrupt in the interval $[t,t+dt]$ is $h$ conditional that bankruptcy hasn't occurred before. We have: 
\begin{multline}
dJ^{pre}={J^{pre}_W}dW+J^{pre}_{WW}(dW)^2+{J^{pre}}_tdt+\\hdt\left[J^{after}(W\times\exp(-\pi_1),t)-J^{pre}(W,t)\right]
\end{multline}
The wealth equation dynamics prior to default is : 
\begin{subequations}
\begin{align}
dW&=W\left(\pi_1\mu dt+\pi_1\sigma dB + (1-\pi_1)rdt\right)\\
dW^2&=W^2\pi_1^2\sigma^2 dt 
\end{align}
\end{subequations}
Substituting into $dJ^{pre}$ we have: 
\begin{multline}
E_t(dJ^{pre})={J^{pre}}_W(W\times\left(\pi_1\mu dt+(1-\pi_1)rdt\right)+\\
\frac{1}{2}J^{pre}_{WW}W^2\pi_1^2\sigma^2 dt +{J^{pre}}_tdt\\+hdt\left[J^{after}(W\times\exp(-\pi_1),t)-J^{pre}(W,t)\right]
\end{multline}
For $J^{pre}$ assume the following form: $J^{pre}(W,t)=f(t)+log(W)$. Then the condition for $\pi_1$ to be optimal is:
\begin{equation}
\frac{\partial}{\partial \pi_1}E_t\left(dJ^{pre}\right)=\left(\mu-r\right)dt-\pi_1\sigma^2 dt-hdt=0
\end{equation}
Solving for $\pi_1$ conclude the formula in the lemma. 
\end{proof}
Assuming $h=0$ we arrive to the Merton's original solution:
\begin{equation}
\pi_1=\frac{\mu-r}{\sigma^2}
\end{equation}
\subsection{$J^{pre}(W,t)$ Expression}
In this section we prove the following expression to $J^{pre}(W,t)$
\begin{theorem}
We have:
\begin{equation}
J^{pre}(t,W)=\frac{-g+\exp(h(t-T)(g-2r)+r(2+h(t-T)}{h}+\log W
\end{equation}
\end{theorem}
\begin{proof}
Assume $J^{pre}(W,t)=f(t)+\log(W)$ and substitute back: $\pi_1=\frac{\mu-r-h}{\sigma^2}$ into 
\begin{multline}
E_t(dJ^{pre})={J^{pre}}_W(W\left(\pi_1\mu dt+(1-\pi_1)rdt\right)+\\
\frac{1}{2}J^{pre}_{WW}W^2\pi_1^2\sigma^2 dt +{J^{pre}}_tdt\\+hdt\left[J^{after}(W\times\exp(-\pi_1),t)-J^{pre}(W,t)\right]
\end{multline}
We obtain the following differential equation for $f(t):$
\begin{equation}
\frac{\left(\mu-r-h\right)^2(2-\sigma^2)}{2\sigma^2}+hr(t-T)+r+f'(t)-hf(t)=0
\end{equation}
To integrate this equation rewrite it as: 
\begin{equation}
f'(t)-hf(t)=-\frac{\left(\mu-r-h\right)^2(2-\sigma^2)}{2\sigma^2}-hr(t-T)-r
\end{equation}
The ODE from last slide will have the form: 
\begin{equation}
f'(t)-hf(t)+hr(t-T)=g-r 
\end{equation}
The initial condition is $f(T)=0$ 
and the the solutions for $f(t)$:
\begin{equation}
f(t)=\frac{-g+\exp(h(t-T)(g-2r)+r(2+h(t-T)}{h}
\end{equation}
Substituting the solution for $f(t)$ back into the expression of $J^{pre}(W,t)$ we conclude the result
\end{proof} 
\par\bigskip\noindent
\section{Conclusion}
 we outlined an approach to address dynamical stochastic problems with absorbing processes. Our main idea is to solve this problem recursively applying Bellman principle and obtaining a system of $2$ HJB's for each process. This idea that hasn't appeared explicitly in the literature to our knowledge leads to a simple analytical solution of Merton's problem for an allocation for stock and cash under the assumption of $\log$ utility. Let us conclude with remarks on our approach that we plan to pursue in subsequent papers: 
\begin{enumerate}
\item While we restrict ourselves to a two state problem this approach can be generalized to more general jump situations. 
\item For aribitrary jump process we can approximate the optimal problem using $2$ state process. For example if $X_1$ can transition to $X_2$ and $X_2$ can transition back to $X_1$  approximate the optimal problem by ignoring the second transition and solving the stochastic control problem ignoring the transition from $X_2$ to $X_1$.) 
An important comment to conclude is to observe that our method of solution is not entirely formalized and the development of proper mathematical framework for our method is desirable. 

\end{enumerate}
{\bf Acknowledgment.} We thank Oleksey Mostovoyi for the interest he expressed in this work. We thank George Yin for very fruitful discussion and remarks that improved our understanding of the subject matter significantly. 

\bibliographystyle{amsplain}

%\noindent $\clubsuit$ Note to author:
%Proceedings articles should be formatted as in reference 1
%above, journal articles as in reference 2 above, and books as in
%reference 3 above.
\end{document}